\documentclass[reqno]{amsart}
\usepackage{amsmath,amsfonts,amscd,amssymb,epsfig,euscript}
\usepackage{tikz}
\usetikzlibrary{calc,decorations.markings}
\usepackage{pdfsync}
\usepackage{amsxtra,mathrsfs}
\theoremstyle{plain}
\newtheorem{theorem}{Theorem}
\newtheorem{lemma}{Lemma}[section]
\newtheorem{proposition}{Proposition}

\theoremstyle{definition}

\newtheorem{remark}{Remark}
\numberwithin{equation}{section}
\newcommand{\field}[1]{\ensuremath{\mathbb{#1}}}
\newcommand{\CC}{\field{C}}

\newcommand{\RR}{\field{R}}

\newcommand{\curly}[1]{\mathscr{#1}}

\newcommand{\cD}{\curly{D}}

\newcommand{\cF}{\curly{F}}

\newcommand{\cH}{\curly{H}}

\newcommand{\cU}{\curly{U}}

\newcommand{\NN}{\field{N}}

\newcommand{\vphi}{\varphi}

\DeclareMathOperator{\im}{\mathrm{Im}}

\usepackage{hyperref}
\hypersetup{colorlinks,citecolor=blue,plainpages=false,hypertexnames=false}
\begin{document}
\title[On the spectral theory of one functional-difference operator]{On the spectral theory of one functional-difference operator from conformal field theory}
\author{Ludwig D. Faddeev}
\address{Sankt Petersburg Department of V.A.~Steklov Mathematical Institute \&
Saint Petersburg State University, Saint Petersburg, Russia} 
\email{faddeev@pdmi.ras.ru}
\author{Leon A. Takhtajan}
\address{Department of Mathematics, Stony Brook University, NY, USA \&
Euler International Mathematical Institute, Saint Petersburg, Russia}
\email{leontak@math.sunysb.edu}
\subjclass[2010]{33D05, 34K06, 39A70}
\keywords{Modular quantum dilogarithm, Weyl operators, functional-difference operator, Schr\"{o}dinger operator, Fourier transform, Casorati determinant, Sokhotski-Plemelj formula, scattering solution, Jost solutions, resolvent of an operator, eigenfunction expansion, Kontorovich-Lebedev transform, scattering theory, scattering operator.}
\thanks{The work of the first author (L.F.) was partially supported by the RFBR grants 14-01-00341 and 13-01-12405-ofi-m and by the program ``Mathematical problems of nonlinear dynamics'' of the Russian Academy of Sciences. The work of the second author (L.T.) was partially supported by the NSF, grant DMS-1005769, and he is grateful to A.M. Polyakov for interesting discussions.}
\begin{abstract}
In the paper we consider a functional-difference operator $H=U+U^{-1}+V$, where $U$ and $V$ are self-adjoint Weyl operators satisfying $UV=q^{2}VU$ with $q=e^{\pi i\tau}$ and $\tau>0$. The operator $H$ has applications in the conformal field theory and in the representation theory of quantum groups. Using modular quantum dilogarithm --- a $q$-deformation of the Euler's dilogarithm --- we define the scattering solution and the Jost solutions, derive an explicit formula for the resolvent of the self-adjoint operator $H$ in the Hilbert space $L^{2}(\RR)$, and prove the eigenfunction expansion theorem. The latter is a $q$-deformation of the well-known Kontorovich-Lebedev transform in the theory of special functions. We also present a formulation of the scattering theory for the operator $H$.
\end{abstract}
\maketitle
\section{Introduction}
The quantum mechanics gave a powerful impetus for the development of the spectral theory of differential operators. In particular, various spectral problems for the Schr\"{o}dinger operator were studied very extensively. Thus in classic papers of I.M. Gelfand, M.G. Krein, B.M. Levitan, V.A. Marchenko and A.Ya. Povzner in the 50s of the last century there have been studied direct and inverse scattering problems for the  Schr\"{o}dinger operator (see surveys  \cite{F59,F74} and references therein). The fundamental role of these papers in the development of the classical integrable systems is well-known.
 
 Formulated in the 80s two-dimensional conformal field theory \cite{BPZ} has stimulated further development of the representation theory of the infinite-dimensional Lie groups and algebras. One of fundamental models of the conformal field theory is the quantum Liouville model, whose discrete version was considered by us more than 30 years ago (see published in 1986 lecture \cite{FT1}). Namely, in the explicit construction of the $L$-operator in \cite{FT1}, the quantum group $\mathrm{SL}_{q}(2,\RR)$ was first introduced. The matrix trace of the $L$-operator --- the functional-difference operator $H$ --- plays important role in the quantization of the Teicm\"{u}ller space \cite{FC,K1} and in the representation theory of the non-compact quantum group $\mathrm{SL}_{q}(2,\RR)$ \cite{PT}. In the notation of Section \ref{W} this operator has the form $H=U+U^{-1}+V$, and acts on functions $\psi(x)$ on the real line by the formula н$$(H\psi)(x)=\psi(x+2\omega')+\psi(x-2\omega') + e^{\frac{\pi i x}{\omega}}\psi(x).$$
Here $\omega$ are $\omega'$ pure imaginary with positive imaginary parts, and the function $\psi(x)$ is assumed analytic in the strip  $|\im z|\leq 2|\omega'|$, $z=x+iy$ (see Sections \ref{W} and \ref{Def} for precise definitions). The operator $H$ is closely related to the representation theory of the quantum group  $\mathrm{SL}_{q}(2,\RR)$ with $q=e^{\pi i\tau}$, where $\tau=\omega'/\omega>0$ (see\cite{PT,FD}). 

The eigenvalue problem for the operator $H$ has the form
\begin{equation} \label{ev-d}
\psi(x+2\omega',\lambda)+\psi(x-2\omega',\lambda) + e^{\frac{\pi i x}{\omega}}\psi(x,\lambda)=\lambda\psi(x,\lambda),
\end{equation}
and is a functional-difference analog of the Schr\"{o}dinger operator with the potential that exponentially decays as  $x\rightarrow-\infty$ and exponentially grows as $x\rightarrow\infty$. Its continuous limit is the equation
\begin{equation}\label{ev-c}
-\tilde{\psi}''(x,\lambda)+e^{2x}\tilde\psi(x,\lambda)=\lambda\tilde\psi(x,\lambda)
\end{equation}
for the modified Bessel functions of the variable $e^{x}$.

 In \cite{K1} the eigenfunction expansion theorem for the operator $H$ in the momentum representation was formulated as formal completeness and orthogonality relations in the distributional sense. The detailed derivation of these relations using the properties of the modular quantum dilogarithm (see Section \ref{Dilog}) was given in \cite{FD}. Nevertheless, the spectral theory of the operator $H$  as unbounded self-adjoint operator on the Hilbert space $L^{2}(\RR)$ has not been considered in the literature.

In the present paper we fill this gap and give a complete analytic study of the functional-difference operator $H$. Namely,  we define the scattering solution and the Jost solutions for equation \eqref{ev-d}, present an explicit formula for the resolvent of the self-adjoint operator $H$ on $L^{2}(\RR)$, and prove the eigenfunction expansion theorem. We also give a formulation of the scattering theory for the operator $H$.

Now let us discuss the content of the paper in more detail. In Section \ref{neccessary} we collect all necessary concepts and notation. Specifically, in Section \ref{W} we define a Weyl pair  $U,V$ of unbounded self-adjoint operators on  $L^{2}(\RR)$ satisfying the relation  $UV=q^{2}VU$, and in Section \ref{Dilog} we present the properties of the modular quantum dilogarithm  $\gamma(z)$, which is a  $q$-deformation of the Euler's dilogarithm and is expressed through the ration of Alekseevski-Barnes double gamma functions. 

In Section \ref{free} we investigate the `free' operator $H_{0}$, formally given by the expression
$$(H_{0}\psi)(x)=\psi(x+2\omega')+\psi(x-2\omega').$$
Thus in Section \ref{Def-0} we define  $H_{0}$ as unbounded self-adjoint operator on $L^{2}(\RR)$ with the domain $D(H_{0})$ and with the absolutely continuous spectrum of multiplicity two filling $[2,\infty)$. In Section \ref{res-free}, using the Fourier transform, we give an explicit expression for the resolvent $R_{0}(\lambda)$ of the operator $H_{0}$ as the integral operator with  the integral kernel $R_{0}(x-y;\lambda)$,      given by formula \eqref{R-0-formula}. Note that unlike the case of  the Schr\"{o}dinger operator, where the resolvent kernel is given by the variation of parameters method using a simple formula  $\theta'(x)=\delta(x)$, where  $\theta(x)$ is the Heaviside function, in the case of the functional-difference operator  $H_{0}$ the main equation for the resolvent
$$R_{0}(x+2\omega';\lambda)+R_{0}(x-2\omega';\lambda)-\lambda R_{0}(x;\lambda)=\delta(x)$$
holds due to the Sokhotski-Plemelj formula
$$\frac{1}{2\pi i}\left(\frac{1}{x-i0}-\frac{1}{x+i0}\right)=\delta(x)$$
from the theory of distributions.

In Section \ref{Def} we study the operator $H$.  Namely, in Section \ref{Kashaev}, following Kashaev \cite{K1} we consider the Fourier transform of equation \eqref{ev-d}, the functional-difference equation of the first order  \eqref{e.v.-m} and its special solution $\hat{\vphi}(p,k)$. It is expressed through the modular quantum dilogarithm, where it is convenient to use the paramet\-rization $\lambda=2\cosh\!\left(\frac{\pi ik}{\omega}\right)$. In Section \ref{scatt} we define a solution $\varphi(x,k)$ of equation \eqref{ev-d} as the inverse Fourier transform of the solution $\hat{\vphi}(p,k)$. In Lemma \ref{phi-properties} we collect necessary properties of the solution $\vphi(x,k)$, which show that it plays the role of the scattering solution of equation \eqref{ev-d}. In particular, for real $x$ and $k$ the solution $\vphi(x,k)$ exponentially decays as $x\rightarrow\infty$ and oscillates as  $x\rightarrow-\infty$. Moreover,  $\vphi(x,k)$ is an entire function of the variable $x$ and
analytically continues to the strip $0<\im k\leq |\omega|$, which corresponds to the values   $\lambda\in \CC\setminus[2,\infty)$.
In Section \ref{Jost} we introduce the Jost solutions  $f_{\pm}(x,k)$ of equation \eqref{ev-d} as the solutions having for real $k$ the asymptotics
$$f_{\pm}(x,k)=e^{\pm2\pi ikx} + o(1)\quad\text{as}\quad x\rightarrow-\infty.$$

It should be noted that unlike differential equation \eqref{ev-c}, which has two linear independent solutions, functional-difference equation \eqref{ev-d} has infinite dimensional space of solutions since one can multiply a solution by a \emph{quasi-constant}  --- holomorphic  $2\omega'$-periodic function of $x$. Therefore determining the Jost solutions is rather nontrivial. Using the similarity with equation \eqref{ev-c} (see Remarks \ref{K-bessel} and \ref{I-bessel}), we define the Jost solutions $f_{\pm}(x,k)$ by integral representation \eqref{f-coor}. Properties of the Jost solutions $f_{\pm}(x,k)$ are given in Lemma \ref{f-properties}. In particular, they admit analytic continuation 
to the strip $0<\im k\leq |\omega|$ and
$$\vphi(x,k)=M(k)f_{+}(x,k)+M(-k)f_{-}(x,k),$$
where the function $M(k)$ is analytic in the strip $0\leq\im k\leq |\omega|$ and is given by explicit formula through the modular quantum dilogarithm  (see Lemma \ref{phi-properties}).

In Section \ref{resolvent-H}, using analytic properties of the solutions $\vphi(x,k)$ and $f_{\pm}(x,k)$, we show that the resolvent  $R_{\lambda}(H)=(H-\lambda I)^{-1}$ of the operator $H$ is defined for $\lambda\notin [2,\infty)$ and is a bounded integral operator on  $L^{2}(\RR)$ with the integral kernel $R(x,y;\lambda)$ given by explicit formula \eqref{R-kernel-2} (see Proposition \ref{resolv}). In Section \ref{expansion} we prove the eigenfunction expansion theorem for scattering solutions of the operator $H$. Namely, computing the jump of the resolvent across the continuous spectrum, we prove (see Theorem  \ref{theorem-1}) that the operator 
$$(\cU \psi)(k)=\int_{-\infty}^{\infty}\psi(x)\vphi(x,k)dx$$
determines the isometric isomorphism of the Hilbert spaces  $L^{2}(\RR)$ and $\cH_{0}=L^{2}([0,\infty),\rho(k)dk)$,
where
$$\rho(k)=\frac{1}{M(k)M(-k)}=4\sinh\!\left(\frac{\pi ik}{\omega}\right)\sinh\!\left(\frac{\pi ik}{\omega'}\right)$$
is the spectral function of the operator $H$.  Here the operator $\cU H\cU^{-1}$ is a multiplication by the function $\lambda=2\cosh\!\left(\frac{\pi ik}{\omega}\right)$ operator on the space $\cH_{0}$, so that $H$ has simple absolutely continuous spectrum filling $[2,\infty)$. 

Comparison of equations  \eqref{ev-d} and \eqref{ev-c} shows (see Remark \ref{K-L}) that the eigenfunction expansion for the operator $H$ is a $q$-analog of the Kontorovich-Lebedev transform, well-known in the theory of Bessel functions. In Section \ref{scattering} we give a formulation of the scattering theory for the operator  $H$ and show that the scattering operator is the operator of multiplication by the function
$$S(k)=\frac{M(-k)}{M(k)}.$$
\section{Basic concepts and notations} \label{neccessary}
Here we present necessary  concepts and notation.
\subsection{Weyl operators} \label{W}
Let $L^{2}(\RR)$ be a Hilbert space of functions, square integrable over the real axis with respect to the Lebesgue measure. Weyl operators in quantum mechanics are unitary on $L^{2}(\RR)$ operators $U(u)$ and $V(v)$ where $u,v\in\RR$,  defined by the formulas 
$$(U(u)\psi)(x)=\psi(x-u),\quad (V(v)\psi)(x)=e^{-ivx}\psi(x),\quad \psi\in L^{2}(\RR)$$
(see, \emph{e.g.}, \cite[Ch.~2]{T1}, where we put the Planck constant $\hbar=1$).
The operators $U(u)$ and $V(v)$ satisfy Hermann Weyl commutation relations 
$$U(u)V(v)=e^{iuv}V(v)U(u).$$
In the representation theory of the quantum group $\mathrm{SL}_{q}(2,\RR)$ it is necessary to use complex $u$ and $v$ so that the Weyl operators $U(u)$ and  $V(v)$ become unbounded self-adjoint operators on $L^{2}(\RR)$. Namely, using classic Weierstrass notation, denote by $2\omega$, $2\omega'$ the generators of a lattice in $\CC$ with the condition that $\im\,\tau>0$, where $\displaystyle{\tau=\frac{\omega'}{\omega}}$, and put $q=e^{\pi i\tau}$. In the representation theory of the quantum group $\mathrm{SL}_{q}(2,\RR)$ the key role is played by the operators $U$ and $V$, formally defined by 
\begin{equation} \label{U-V}
(U\psi)(x)=\psi(x+2\omega'), \quad (V\psi)(x)=e^{\frac{\pi i x}{\omega}}\psi(x)
\end{equation}
and satisfying the relation
\begin{equation} \label{U-V-q}
UV=q^{2}VU
\end{equation}
on the common domain of $U$ and $V$. Especially interesting is the representation theory of the quantum group $\mathrm{SL}_{q}(2,\RR)$ which corresponds to the cases  $q\in\RR$ and $|q|=1$. In the first case $\omega'\in i\RR$, $\omega\in\RR$ and $0<q<1$, which corresponds to the rectangular period lattice, whereas in the second case the half-periods $\omega$, $\omega'$ are pure imaginary and theory of elliptic functions breaks down.

It is the latter case that arises in the application to the conformal field theory, and here we consider the Weyl pair   
$U, V$ with $|q|=1$. This corresponds to the case when the half-periods $\omega$, $\omega'$ are pure imaginary with positive imaginary parts.
It is convenient to use a normalization
\begin{equation} \label{norm}
\omega\omega'=-\frac{1}{4},
\end{equation}
which is assumed to hold throughout the paper. In the literature on the quantum Liouville theory, it is customary, following the work of A.B. and Al.B. Zamolodchikovs \cite{ZZ-1}, to use a parametrization $\omega=\frac{i}{2b}, \omega=\frac{ib}{2}$, where $\tau=b^{2}$ and $b>0$. 

The operators $U$ and $V$ defined by  \eqref{U-V} are unbounded self-adjoint operators on  $L^{2}(\RR)$. This follows from the general spectral theorem of von Neumann, since they are real-valued functions of the self-adjoint operators $P$ and $Q$ in quantum mechanics, $U=e^{2i\omega'\!P}$ and $V=e^{\frac{\pi i Q}{\omega}}$, where $\displaystyle{P=-i\frac{d}{dx}}$, and  $Q$ is the operator of multiplication by the independent variable $x$. 

The Weyl operators $U$ and $V$ can be also defined directly. Namely, $U$ is a self-adjoint operator in $L^{2}(\RR)$ with the domain 
$$D(U)=\{\psi(x)\in L^{2}(\RR) : e^{-\frac{\pi i p}{\omega}}\hat{\psi}(p)\in L^{2}(\RR)\},$$ 
where
$$\hat{\psi}(p)=\cF(\psi)(p)=\int_{-\infty}^{\infty}\psi(x)e^{-2\pi ipx}dx$$
is the Fourier transform\footnote{We use the normalization of the Fourier transform adopted in the analytic number theory.} in $L^{2}(\RR)$.  
Equivalently, the domain $D(U)$ consists of functions $\psi(x)$ which admit analytic continuation into the strip $\{z=x+iy\in\CC: 0< y <2|\omega'|\}$ such that  $\psi(x+iy)\in L^{2}(\RR)$ for all $0\leq y<2|\omega'|$ and the following limit 
$$\psi(x+2\omega'-i0)=\lim_{\varepsilon\rightarrow 0^{+}}\psi(x+2\omega'-i\varepsilon)$$
exists in a sense of the convergence in $L^{2}(\RR)$. Here for $\psi\in D(U)$ we have $(U\psi)(x)=\psi(x+2\omega'-i0)$. The domain $D(U^{-1})$ of the inverse operator $U^{-1}$ is defined similarly and  $(U^{-1}\psi)(x)=\psi(x-2\omega'+i0)$. The domain $D(V)$ of the self-adjoint operator $V$ consists of functions $\psi(x)\in L^{2}(\RR)$ such that $e^{\frac{i\pi x}{\omega}}\psi(x)\in L^{2}(\RR)$. 
Therefore we have
$$U^{-1}=\cF^{-1}V\cF,$$
where the inverse Fourier transform is given by
$$\psi(x)=\int_{-\infty}^{\infty}\hat{\psi}(p)e^{2\pi ipx}dp.$$
\begin{remark} \label{dual} In the representation theory important role is played by the modular double of the quantum group
 $\mathrm{SL}_{q}(2,\RR)$, introduced in \cite{Faddeev-1}. Its principal series representations are realized in  $L^{2}(\RR)$ and side by side  with the operators $U$ and $V$ use the dual Weyl operators  $\check{U}$ and $\check{V}$. They satisfy dual to  \eqref{U-V-q} relation
$$\check{U}\check{V}=\check{q}^{2}\check{V}\check{U}, \quad \check{q}=e^{\pi i/\tau}$$
and are given by the formulas
$$(\check{U}\psi)(x)  = \psi(x+2\omega),\quad (\check{V}\psi)(x)  = e^{\frac{\pi i x}{\omega'}}\psi(x),$$
obtained from \eqref{U-V} by interchanging the half-periods $\omega$ and $\omega'$. 
\end{remark}

\subsection{Modular quantum dilogarithm} \label{Dilog}
Put

\begin{equation} \label{dilog}
\gamma(z)=\exp\left\{ -\frac{1}{4}\int_{-\infty}^{\infty}\frac{e^{itz}}{\sin\omega t\sin\omega't}\frac{dt}{t}\right\},
\end{equation}
\noindent
where $|\im z|<|\omega|+|\omega'|$ and contour of integration bypasses the singularity at $t=0$ from above. The function $\gamma(z)$ plays a fundamental role in the definition of the modular double of the quantum group $\mathrm{SL}_{q}(2,\RR)$, given by the first author in \cite{Faddeev-1}. It was later  given a name \emph{modular quantum dilogarithm}. Here the adjective ``modular'' reflects the invariance of the function $\gamma(z)$ under the interchanged of $\omega$ and $\omega'$, i.e., under the change of $\tau$ by $1/\tau$. The term ``quantum dilogarithm'' in the name of  $\gamma(z)$ is related to its asymptotic as  $\tau=b^{2}\rightarrow 0$, which for real $z$ is easy to get from the representation \eqref{dilog},
$$\gamma\left(\frac{z}{b}\right)=\exp\left\{\frac{1}{2\pi\tau}\,\mathrm{Li}_{2}(-e^{-2\pi z})+O(1)\right\}\quad\text{при}\quad \tau\rightarrow 0,$$
where
$$\mathrm{Li}_{2}(z)=\sum_{n=1}^{\infty}\frac{z^{n}}{n^{2}}$$
is the Euler's dilogarithm. 

\begin{remark} The function $\gamma(z)$ has interesting history. It appears in the number theory under the name ``double sine''  \cite{Sh, Kurokawa}, in the theory of quantum integrable systems of Calogero-Moser under the name ``hyperbolic gamma function''  \cite{Ruij}, it plays the role of $S$-matrix in the quantum nonlinear $\sigma$-model \cite{ZZ-2} and appears in form-factors for the quantum Sine-Gordon model \cite{Smirnov}. The function $\gamma(z)$ is expressed through the ratio of gamma-functions of the second order, introduced by Barnes \cite{Barnes} in 1899 and actually investigated earlier in 1889 in V.P. Alekseevski's thesis \cite{Al}.
\end{remark}

We will be using the following properties of the modular quantum dilogarithm (see \cite{K1, Volkov}).
\begin{itemize}
\item[QD1)] The function $\gamma(z)$ admits a meromorphic continuation to the whole complex $z$ plane with the poles at  $z=-(2m+1)\omega - (2n+1)\omega'$ with integer $m,n\geq 0$. Here 
$$\gamma(z-\omega'')=\frac{c}{z} + O(1),\quad\text{as}\quad z\rightarrow 0,$$
where $\omega''=\omega+\omega'$ and
$$ c=\frac{e^{i(\frac{\pi}{4}-\beta)}}{2\pi},\quad \beta=\frac{\pi}{12}\left(\tau+\frac{1}{\tau}\right).$$
\item[QD2)]  The function $\gamma(z)$ satisfies the following difference equations
\begin{align*}
\gamma(z+\omega) &=(1+e^{-\frac{\pi i z}{\omega'}})\gamma(z-\omega), \\
\gamma(z+\omega') &=(1+e^{-\frac{\pi i z}{\omega}})\gamma(z-\omega').
\end{align*}
\item[QD3)]  The reflection formula
$$\gamma(z)\gamma(-z)=e^{i\beta+i\pi z^{2}},$$
so that $\gamma(z)$ has zeros at $z=(2m+1)\omega+(2n+1)\omega'$ with integer $m,n\geq 0$.
\item[QD4)]  The reality property
$$\overline{\gamma(z)}=\frac{1}{\gamma(\bar{z})}.$$ 
\item[QD5)]   
The function  $\gamma(z)$ has the following asymptotic 
$$\gamma(z)=1+o(1)$$
as $z\rightarrow\infty$ such that $|\arg z|<\frac{\pi}{2}-\delta$, uniformly in $z$ for each $0<\delta<\frac{\pi}{2}$. 
\end{itemize}

\section{The operator $H_{0}$} \label{free}
Here we consider the free operator $H_{0}=U+U^{-1}$, which formally acts on functions $\psi(x)$ on the real line by the formula
\begin{equation}\label{H-0}
(H_{0}\psi)(x)=\psi(x+2\omega')+\psi(x-2\omega'),
\end{equation}
where it is understood the the function $\psi(x)$ is analytic in the stirp $|\im z|\leq 2|\omega'|$, $z=x+iy$. Obviously when $2\omega'=ib\rightarrow 0$ the operator $b^{-2}(H_{0}-2I)$ turns into the operator $\displaystyle{-\frac{d^{2}}{dx^{2}}}$.

\subsection{The domain} \label{Def-0} Formula \eqref{H-0} determines unbounded self-adjoint operator  $H_{0}$ on $L^{2}(\RR)$ with the domain  $D(H_{0})$, consisting of functions $\psi(x)$ which admit analytic continuation to the stirp $\{z=x+iy\in\CC: |y| <2|\omega'|\}$ and such that  $\psi(x+iy)\in L^{2}(\RR)$ for all  $|y|<2|\omega'|$ and following limits  
$$\psi(x+2\omega'-i0)=\lim_{\varepsilon\rightarrow 0^{+}}\psi(x+2\omega'-i\varepsilon)\quad\text{and}\quad \psi(x-2\omega'+i0)=\lim_{\varepsilon\rightarrow 0^{+}}\psi(x-2\omega'+i\varepsilon)$$
exist in the sense of the convergence in $L^{2}(\RR)$. Here for $\psi\in D(H_{0})$ formula \eqref{H-0} is understood as $(H_{0}\psi)(x)=\psi(x+2\omega'-i0)+\psi(x-2\omega'+i0)$.

In the `momentum representation' the operator $\hat{H}_{0}=\cF H_{0}\cF^{-1}$ is the operator of multiplication by the function $2\cosh(\frac{\pi ip}{\omega})$ and is naturally self-adjoint. Thus the domain  $D(H_{0})$ of the operator $H_{0}$ can be equivalently defined as 
$$D(H_{0})=\left\{\psi(x)\in L^{2}(\RR) : \int_{-\infty}^{\infty}\cosh^{2}\!\left(\frac{\pi ip}{\omega}\right)|\hat{\psi}(p)|^{2}dp<\infty\right\}$$
and is a `hyperbolic analog' of the Sobolev space $W^{2,2}(\RR)$. 
\subsection{The resolvent of the operator $H_{0}$} \label{res-free}
In the momentum representation the operator  
$$R_{0}(\lambda)=(H_{0}-\lambda)^{-1}$$
is a multiplication operator by the function  $(2\cosh(\frac{\pi ip}{\omega})-\lambda)^{-1}$,  and for $\lambda\in\CC\setminus [2,\infty)$ is a bounded operator on $L^{2}(\RR)$. Since the function $2\cosh(\frac{\pi ip}{\omega})$ is a 2-to-1 map of the real axis $-\infty<p<\infty$  onto  $[2,\infty)$, the spectrum of the operator $H_{0}$ is absolutely continuous and fills the semi-infinite interval  $[2,\infty)$ with the multiplicity $2$.

In the `coordinate representation' the operator $R_{0}(\lambda)$ for $\lambda\in\CC\setminus [2,\infty)$ is an integral operator on $L^{2}(\RR)$ with the kernel which depends on the difference,
\begin{equation} \label{R-0-def}
(R_{0}(\lambda)\psi)(x)=\int_{-\infty}^{\infty}R_{0}(x-y;\lambda)\psi(y)dy,
\end{equation}
where
\begin{equation} \label{R-0-kernel}
R_{0}(x;\lambda)=\int_{-\infty}^{\infty}\frac{e^{2\pi ipx}}{2\cosh(\frac{\pi ip}{\omega})-\lambda}dp.
\end{equation}
In the what follows we will be using a convenient parametrization
\begin{equation} \label{l-k}
\lambda=2\cosh\!\left(\frac{\pi i k}{\omega}\right),
\end{equation}
 under which the resolvent set $\CC\setminus [2,\infty)$ turns into the `physical sheet' --- the strip $0<\im k\leq |\omega|$,
and the continuous spectrum $[2,\infty)$ is twice covered by the real axis $-\infty<k<\infty$.  In parametrization \eqref{l-k} the integral in\eqref{R-0-kernel} is easily evaluated by the residue theorem and we obtain
\begin{equation} \label{R-0-formula}
R_{0}(x;\lambda) =\frac{\omega}{\sinh(\frac{\pi ik}{\omega})}\left(\frac{e^{-2\pi ikx}}{1-e^{-4\pi i\omega x}}+\frac{e^{2\pi ikx}}{1-e^{4\pi i\omega x}} \right).
\end{equation}
Note that the function $R_{0}(x;\lambda)$ is regular at $x=0$. From \eqref{R-0-formula} we immediately conclude that for $0<\im k\leq |\omega|$ the following estimate holds 
$$|R_{0}(x;\lambda)|\leq Ce^{-2\pi\im k|x|},$$
where $C>0$ is a constant\footnote{Here and in what follows we denote different constants by $C$.}, so that formulas \eqref{R-0-def} and \eqref{R-0-formula} for $\lambda\notin [2,\infty)$ do determine a bounded operator on $L^{2}(\RR)$.

The eigenvalue equation
\begin{equation} \label{H-0-eqn}
\psi(x+2\omega',k)+\psi(x-2\omega',k)=2\cosh\!\left(\frac{\pi ik}{\omega}\right) \psi(x,k)
\end{equation}
has solutions $f_{-}(x,k)=e^{-2\pi ikx}$ and $f_{+}(x,k)=e^{2\pi ikx}$, which are analogs of the Jost solutions in the theory of one-dimensional Schr\"{o}dinger equation. In terms of the Jost solutions formula  \eqref{R-0-formula} takes the form
\begin{equation} \label{R-0-kernel-2}
R_{0}(x-y;\lambda)=\frac{2\omega}{C(f_{-},f_{+})(k)}\left(\frac{f_{-}(x,k)f_{+}(y,k)}{1-e^{\frac{\pi i}{\omega'}(x-y)}}+\frac{f_{-}(y,k)f_{+}(x,k)}{1-e^{-\frac{\pi i}{\omega'}(x-y)}} \right),
\end{equation}
where 
$$C(f,g)(x,k)=f(x+2\omega',k)g(x,k)-f(x,k)g(x+2\omega',k)$$
is the so-called \emph{Casorati determinant} (difference analog of the Wronskian) of solutions of the functional-difference equation\eqref{H-0-eqn}.  It is $2\omega'$-periodic function of $x$ and for the Jost solutions
$C(f_{-},f_{+})(x,k)=2\sinh(\frac{\pi ik}{\omega})$.
\begin{remark} \label{H-0-res} Using formula \eqref{R-0-kernel-2}, one can check directly that the integral operator  \eqref{R-0-def} is the inverse of the operator $H-\lambda I$ for $\lambda\in\CC\setminus [2,\infty)$. Indeed, for a smooth function $g(x)$ with compact support it is easy to show that
$$\psi(x)=\int_{-\infty}^{\infty}R_{0}(x-y;\lambda)g(y)dy\in D(H_{0})$$
and $(H_{0}-\lambda I)\psi=g$. The last statement reduces to the following equation
\begin{equation} \label{delta}
R_{0}(x+2\omega'-y-i0;\lambda)+R_{0}(x-2\omega'-y+i0;\lambda)-\lambda R_{0}(x-y;\lambda)=\delta(x-y),
\end{equation}
understood in the distributional sense. Because the functions $f_{\pm}(x,k)$ satisfy equation \eqref{H-0-eqn}, оthe distribution in the left hand side of  \eqref{delta} has support only at $x=y$, and its singular part is the same as the singular part of the distribution 
\begin{gather*}
-\frac{2\omega\omega'}{\pi iC(f_{-},f_{+})(k)}\left(\frac{f_{-}(x+2\omega',k)f_{+}(y,k)-f_{-}(y,k)f_{+}(x+2\omega',k)}{x-y-i0} \right.\\
\left.+\frac{f_{-}(x-2\omega',k)f_{+}(y,k)+f_{-}(y,k)f_{+}(x+2\omega',k)}{x-y+i0}\right)
\end{gather*}
in the neighborhood of $x=y$. The latter is
$$-\frac{2\omega\omega'}{\pi i}\left(\frac{1}{x-y-i0}-\frac{1}{x-y+i0} \right)=\delta(x-y),$$
where we have used the definition of the Casorati determinant, normalization \eqref{norm} and the Sokhotski-Plemelj formula.
\end{remark}
\begin{remark}\label{schrodinger} It is instructive to compare formula \eqref{R-0-kernel-2} for the resolvent of the operator $H_{0}$ with that for the one-dimensional Schr\"{o}dinger equation. The latter formula is  (see, \emph{e.g.}, \cite[Ch.1, \S1]{F74} and \cite[Ch. 3]{T1})
\begin{equation} \label{res-schr}
G(x,y;\lambda)=\frac{1}{W(k)}(f_{-}(x,k)f_{+}(y,k)\theta(y-x)+f_{-}(y,k)f_{+}(x,k)\theta(x-y)),
\end{equation}
where $\lambda=k^{2}$, $f_{-}(x,k)$ and $f_{+}(x,k)$ are, respectively, the Jost solutions and  $-\infty$ and $\infty$,  and  $W(k)$ is their Wronskian. The key role in verifying the analog of formula \eqref{delta} plays the relation $\theta'(x)=\delta(x)$, where  $\theta(x)$ is the Heaviside function, $\theta(x)=1$ when $x>0$ and $\theta(x)=0$ when $x<0$. Formula  \eqref{R-0-kernel-2} has a remarkable similarity with \eqref{res-schr}, where instead of $\theta(x)$ a smoothed analog of the Heaviside function
$$\theta_{\omega'}(x)=\frac{1}{1-e^{-\frac{\pi ix}{\omega'}}}$$
is used. Here the analog of the formula $\theta'(x)=\delta(x)$ is the formula 
$$\theta_{\omega'}(x+2\omega'-i0)-\theta_{\omega'}(x+2\omega'+i0)=2\omega'\delta(x),$$
which is equivalent to the Sokhotski-Plemelj formula.Heaviside
\end{remark}
\begin{remark} As noted by A.M. Polyakov, the function $\theta_{\omega'}(x)$,
after the identification of  $x$ with the energy $\epsilon$, and $\frac{\pi i}{\omega'}=\frac{2\pi}{b}$ --- with the inverse temperature $\frac{1}{kT}$, coincides with the one-particle partition function  $\displaystyle{\mathcal{Z}=\left(1-e^{-\frac{\epsilon}{kT}}\right)^{-1}}$ in the Bose-Einstein statistics.
\end{remark}

\section{The operator $H$} \label{Def} 

The operator $H=H_0 + V$ is given by the following formal functional-difference formula
$$(H\psi)(x)=\psi(x+2\omega')+\psi(x-2\omega') + e^{\frac{\pi i x}{\omega}}\psi(x),$$
defined on $D(H_0)\cap D(V)$. In particular, $H$ is defined and symmetric on the dense in $L^{2}(\RR)$ domain $\cD\subset D(H_0)\cap D(V)$, which consists of linear combination of the functions  $p(x)e^{-x^{2}+cx}$, where $p(x)$ is a polynomial and $c\in\CC$. The domain $\cD$ is invariant for the operator  $H$. Below we will show that the operator $H$ is essentially self-adjoint on the domain $\cD$  and its unique self-adjoint extension, which we continue to denote by $H$, has a simple absolutely continuous spectrum filling $[2,\infty)$. As in the case of the free operator  $H_0$, for the operator $H$ we use parametrization  \eqref{l-k} and consider the following problem for the generalized eigenfunctions 
\begin{equation}\label{ev-H}
\psi(x+2\omega',k)+\psi(x-2\omega',k) + e^{\frac{\pi i x}{\omega}}\psi(x,k)=2\cosh\!\left(\frac{\pi ik}{\omega}\right)\psi(x,k).
\end{equation}

\subsection{The momentum representation and Kashaev's wave function} \label{Kashaev}

In the momentum representation the eigenfunction problem for the operator $H$ is the following first order functional-difference equation
\begin{equation} \label{e.v.-m}
\hat{\psi}(p+2\omega',k)=2\left(\!\cosh\!\left(\frac{\pi ik}{\omega}\right)-\cosh\left(\frac{\pi ip}{\omega}\right)\right)\hat{\psi}(p,k),\quad p\in\RR,
\end{equation}
where $\hat\psi =\cF(\psi)$. Remarkably, the general solution of equation \eqref{e.v.-m} (up to a multiplication by a quasi-constant!) is given explicitly through the modular quantum dilogarithm 
\begin{equation} \label{m-e.f}
\hat{\psi}(p,k)=c(k)e^{-\pi i(p-\omega'')^{2}}\gamma(p+k-\omega'')\gamma(p-k-\omega''),\quad 0\leq \im k\leq |\omega|,
\end{equation}
where the constant $c(p,k)$ will be chosen just below. For real $k$ the product of $\gamma$-functions is singular at $p=\pm k$ and is understood as the distribution $\gamma(p+k-\omega''+i0)\gamma(p-k-\omega''+i0)$.

The fundamental role of the generalized solution  \eqref{m-e.f} of equation \eqref{e.v.-m} was revealed in the paper \cite{K1}, and we call it  \emph{Kashaev's wave function}.
The distribution $\hat{\psi}(p,k)$ has the following asymptotics
\begin{equation} \label{as}
\hat\psi(p,k)=\begin{cases} c(k)e^{-\pi i(p-\omega'')^{2}}(1+o(1)) & \text{as}\quad p\rightarrow\infty, \\
c(k)e^{\pi i(p-\omega'')^{2} +2 i\beta + 2\pi ik^{2}}(1+o(1)) & \text{as}\quad p\rightarrow-\infty
\end{cases}
\end{equation}
and exponentially decays at large $p$,
\begin{equation} \label{p-large}
|\hat{\psi}(p,k)|=|c(k)|\exp\{-2\pi |p||\omega''|\}(1+o(1))\quad\text{as}\quad|p|\rightarrow\infty.
\end{equation}
Putting in \eqref{m-e.f}
\begin{equation} \label{c}
c(k)= e^{-i\beta-\pi ik^{2}}
\end{equation}
and denoting the corresponding solution by $\hat{\vphi}(x,k)$, we obtain important property
\begin{equation} \label{m-even}
\overline{\hat{\vphi}(p,k)}=\hat{\vphi}(-p,-\bar{k}).
\end{equation}
Moreover, for real $k$ we have  $\hat{\vphi}(p,-k) =\hat{\vphi}(p,k)$.

\begin{remark} \label{K-bessel} It is instructive to compare equation  \eqref{e.v.-m} with equation \eqref{ev-c} in the momentum representation, which has the form 
\begin{equation} \label{ev-m}
\hat{\tilde{\psi}}\!\left(p+ \frac{i}{\pi},k\right)=4\pi^2(k^2-p^2)\hat{\tilde{\psi}}(p,k),
\end{equation}
where we put $\lambda=(2\pi k)^2$. Its solution is given by the product of the Euler gamma functions
\begin{equation} \label{Gamma}
\hat{\tilde{\psi}}(p,k)=2^{-2\pi i p-2}\Gamma(\pi i(p+k))\Gamma(\pi i(p-k)),
\end{equation}
and a general solution is obtained by multiplication by a quasi-constant --- a periodic function with the period $i/\pi$.
Performing the inverse Fourier transform and putting  $s=-2\pi ip$, we get the Mellin-Barnes representation for the modified Bessel function of the second kind 
\begin{align} \label{K-B}
K_{\nu}(e^x) 
=\frac{1}{8\pi i}\int_{\sigma-i\infty}^{\sigma+i\infty}\left(\frac{e^{x}}{2}\right)^{-s}\Gamma\left(\frac{s-\nu}{2}\right)\Gamma\left(\frac{s+\nu}{2}\right)ds,
\end{align}
where $\nu=2\pi ik$ and $\sigma=\mathrm{Re}\,s>|\mathrm{Re}\,\nu|$ (see.~\cite[Ch. 7, formula (27)]{BE} for the Mellin transform of the function  $K_{\nu}(z)$). When $\mathrm{Re}\,\nu=0$,  the integration goes over the imaginary axis $\sigma=0$ bypassing the poles at $s=\pm \nu$ in the half-plane $\mathrm{Re}\,s>0$. The function $K_{\nu}(e^x)$ is entire function of $x$.
\end{remark}

\begin{remark} \label{I-bessel}
Modified Bessel functions of the first kind $I_{\nu}(e^{x})$ and $I_{-\nu}(e^{x})$,  where $\lambda=-\nu^{2}$ (see Remark \ref{K-bessel}) are also solutions of equation \eqref{ev-c}. Here
$$W(I_{-\nu}, I_{\nu})(x)=I_{-\nu}(x)I_{\nu}'(x)-I_{-\nu}'(x)I_{\nu}(x)=\frac{2\sin\pi\nu}{\pi}.$$
The function $I_{\nu}(e^{x})$ is obtained by multiplying solution  \eqref{Gamma} of equation \eqref{ev-m} by the quasi-constant $\displaystyle{\frac{(e^{-\pi i\nu }-e^{-\pi is})}{\pi i}}$, where $s=-2\pi ip$ and by replacing the contour of integration in  \eqref{K-B} by the countour $C$ on Fig. 1.  
\vspace{5mm}
\begin{center}
\begin{tikzpicture}[thick,scale=0.6]
\draw(-5,-3) -- (3.2,-3);
\draw(3.2,-3) -- (3.2,3);
\draw(3.2,3) -- (-5,3);
\node at (2,1.5) {$\dot\nu$};
\node at (-2,-1.5) {$\dot-\nu$};
\node at (3.8,0) {$C$};
\node at (-3,3) {$<$};
\node at (-3,-3) {$>$};
\node at (0,-4) {$\text{Fig. 1}$};
\end{tikzpicture}
\end{center}
\noindent
As the result we get an integral representation
\begin{equation} \label{I-B}
I_{\nu}(e^{x})=-\frac{1}{8\pi^{2}}\int\limits_{C}\left(\frac{e^{x}}{2}\right)^{-s}\Gamma\left(\frac{s-\nu}{2}\right)\Gamma\left(\frac{s+\nu}{2}\right)(e^{-\pi i\nu }-e^{-\pi is})ds.
\end{equation}
Shifting the contour of integration $C$ to the left to $-\infty$, we obtain the standard representation of мы получаем стандартное $I_{\nu}(e^{x})$ as the power series in variable $e^{x}$. The factor $e^{-\pi i\nu }-e^{-\pi is}$ ensures that there are no poles at $s=\nu+2-2n$, $n\in\NN$. The function $I_{\nu}(e^x)$ is entire function of $x$ and $I_{\nu}(e^{x+\pi i})=e^{\pi i\nu}I_{\nu}(e^{x})$. From \eqref{K-B} and \eqref{I-B} we get
\begin{equation} \label{K-I}
K_{\nu}(e^{x})=\frac{\pi}{2\sin\pi\nu}\left(I_{-\nu}(e^{x})-I_{\nu}(e^{x})\right),
\end{equation}
and also
$$K_{\nu}(e^{x+\pi i})=\frac{\pi}{2\sin\pi\nu}\left(e^{-\pi i\nu}I_{-\nu}(e^{x})-e^{\pi i\nu}I_{\nu}(e^{x})\right),$$
so that
\begin{equation}\label{I-K}
I_{\nu}(e^{x})=\frac{1}{\pi i}\left(e^{-\pi i\nu}K_{\nu}(e^{x})-K_{\nu}(e^{x+\pi i})\right). 
\end{equation}
The function $I_{\nu}(e^x)$  has the following asymptotic as  $ x\rightarrow-\infty$,
$$I_{ik}(e^{x})=\frac{2^{-ik}}{\Gamma(1+ik)}\left(e^{ikx}+o(1) \right),$$
and grows as a double exponent as $x\rightarrow\infty$, whereas $K_{\nu}(e^{x})=O(e^{-e^{x}})$ as $x\rightarrow\infty$.
\end{remark}

\subsection{The scattering solution} \label{scatt}
For real $x$ and $k$ put
\begin{equation} \label{phi-coor}
\varphi(x,k)=\int_{-\infty}^{\infty}\hat{\vphi}(p,k)e^{2\pi i px}dp,
\end{equation}
where $\hat\vphi(p,k)$ is given by formulas  \eqref{m-e.f} and \eqref{c}. Remark \ref{K-bessel} shows that the function $\vphi(x,k)$ plays the role of $q$-deformed modified Bessel function $K_{ik}(e^x)$. It follows from asymptotic  \eqref{as} that the Kashaev's wave function exponentially decays as $|\mathrm{Re}\, p|\rightarrow\infty$ along the lines $\im p=\sigma<|\omega''|$, therefore
\begin{equation} \label{phi-coor-2}
\varphi(x,k)=\int_{-\infty+i\sigma}^{\infty+i\sigma}\hat{\vphi}(p,k)e^{2\pi i px}dp,
\end{equation}
This formula determines the function $\varphi(x,k)$ for real $x$ and $k$ in the physical strip $0<\im k\leq |\omega|$, где $|\omega|<\sigma<|\omega''|$.

Analytic properties of the function are $\phi(x,k)$ described in the following lemma. 
\begin{lemma} \label{phi-properties}  The following statements hold.
\begin{itemize}
\item[(i)]  For real $x$ and $k$ the function $\vphi(x,k)$ has the following asymptotic
$$\vphi(x,k)=M(k)e^{2\pi ikx}+M(-k)e^{-2\pi ikx} + o(1)\quad\text{as}\quad x\rightarrow -\infty,$$
where
$$M(k)=e^{i(\beta+\frac{\pi}{4}) - 2\pi ik(k-\omega'')}\gamma(2k-\omega'').$$
Here $\overline{M(k)}=M(-k)$ and
$$\frac{1}{|M(k)|^{2}}=4\sinh\!\left(\frac{\pi i k}{\omega}\right)\sinh\!\left(\frac{\pi i k}{\omega'}\right).$$
\item[(ii)] For real $x$ the function $\varphi(x,k)$ admits analytic continuation into the physical strip  $0<\im k\leq |\omega|$ and satisfies the reality condition
$$\overline{\varphi(x,k)}=\varphi(x,-\bar{k}).$$
For real $x$ and  $k$ the function $\vphi(x,k)$ is a real-valued even function of $k$.
\item[(iii)] For fixed $k$ in the physical strip the function $\varphi(x,k)$ extends to an entire function of the complex variable $x$ and satisfies the equation 
$$\vphi(x+2\omega',k)+\vphi(x-2\omega',k)+e^{\frac{\pi i x}{\omega}}\vphi(x,k)=2\cosh\!\left(\frac{\pi ik}{\omega}\right)\vphi(x,k).$$
\item[(iv)] The following estimates hold:
\begin{equation*}
|\vphi(x,k)|\leq Ce^{-2\pi\kappa x},
\end{equation*}
uniformly in $-\infty<x\leq a$, where $0\leq\kappa=\mathrm{Im}\,k\leq |\omega|$ and
\begin{equation*}
|\vphi(x,k)|\leq C e^{-2\pi(|\omega|+|\omega'|)x},\quad 
|\vphi(x\pm 2\omega',k)|\leq C e^{2\pi(|\omega'|-|\omega|)x},
\end{equation*}
uniformly in $a\leq x<\infty$.
\end{itemize}
\end{lemma} 
\begin{proof} Shifting for real $x$ the contour of integration in formula \eqref{phi-coor} to the lower half-plane and passing the poles of the integrand at $p=\pm k$,  we get the first formula in (i). The formulas for the coefficient $M(k)$ follow from the properties 
QD1)--QD4) of the modular quantum dilogarithm. The statement (ii) directly follows from the analytic properties of the modular quantum dilogarithm and property \eqref{m-even}. In particular, $\overline{M(k)}=M(-\bar{k})$.

To proof (iii), deform the contour of integration in the integral representation \eqref{phi-coor-2} to the contour $L$ by replacing the semi-infinite intervals  $-\infty<\mathrm{Re}\,p\leq-|\mathrm{Re}\,k|-1$ and $|\mathrm{Re}\,k|+1\leq\mathrm{Re}\,p<\infty$ on the line $\im p=\sigma$ by the rays $p=-|\mathrm{Re}\,k|-1 +i\sigma +e^{\frac{\pi i}{4}}t$, where $-\infty<t\leq 0$ and $p=|\mathrm{Re}\,k|+1+ i\sigma +e^{-\frac{\pi i}{4}}t$, where $0\leq t<\infty$ (on Fig. 2 the contour $L$ is shown for real  $k$).
  
\vspace{5mm}
\begin{tikzpicture}[thick]
\draw (-6,-2) -- (-4,0);
\draw(-4,0) -- (-2.3,0);
\draw (-2.3,0) to [out=90,in=180] (-2,0.3) to [out=0,in=90] (-1.7,0);
\draw (-1.7,0) -- (1.7,0);
\draw (1.7,0) to [out=90,in=180] (2,0.3) to [out=0,in=90] (2.3,0);
\draw (2.3,0) -- (4,0);
\draw (4,0) -- (6,-2);
\node at (0,.3) {$L$};
\node at (-3,0) {$>$};
\node at (3,0) {$>$};
\node at (-2,0) {$\cdot$};
\node at (2,0) {$\cdot$};
\node at (-2.13,-.25) {$-k$};
\node at (2,-.25) {$k$};
\node at (0,-3) {$\text{Fig. 2}$};
\end{tikzpicture}

\noindent
Thus
\begin{equation} \label{phi-coor-3}
\varphi(x,k)=\int\limits_{L}\hat{\vphi}(p,k)e^{2\pi i px}dp,
\end{equation}
and it follows from the property QD5) of the function $\gamma(z)$ that on the contour $L$ the integrand in \eqref{phi-coor-3} decays as  $e^{-\pi t^{2}}$ when $t\rightarrow\pm\infty$,  so that formula  \eqref{phi-coor-3} determines  $\vphi(x,k)$ as entire function of the variable $x$. The difference equation for  $\vphi(x,k)$ is obtained from \eqref{e.v.-m} by the means of the Fourier transform. Finally, it is standard to deduce the estimates  (iv) from the integral representation \eqref{phi-coor-3} using asymptotics к QD5) of the function $\gamma(z)$ and the steepest descent method. We leave a detailed derivation to the reader.
\end{proof}

\begin{remark} \label{dual-2} The function $\vphi(x,k)$ is invariant under the interchange of  $\omega$ and $\omega'$ and satisfies the dual equation $\check{H}\vphi=\check{\lambda}\vphi$, where $\check{H}=\check{U}+\check{U}^{-1}+ \check{V}$ (see Remark \ref{dual}) and $\check\lambda=2\cosh\!\left(\frac{\pi ik}{\omega'}\right)$.
\end{remark}

\subsection{The Jost solutions} \label{Jost}
When $x\rightarrow-\infty$ equation \eqref{ev-H} takes the free form  \eqref{H-0-eqn} and it is natural to assume that equation \eqref{ev-H} has the Jost solutions  --- solutions $f_{\pm}(x,k)$ with the following asymptotics
\begin{equation} \label{f-asym}
f_{\pm}(x,k)=e^{\pm2\pi ikx} + o(1)\quad\text{as}\quad x\rightarrow-\infty.
\end{equation}

Here we prove the existence of such solutions. The starting point for us is the comparison of equation \eqref{ev-H} with equation \eqref{ev-c}, where the role of the scattering solution is played by the function $\tilde{\vphi}(x,k)=K_{ik}(e^{x})$. Corresponding Jost solutions are the functions  $\tilde{f}_{+}(x,k)=2^{ik}\Gamma(1+ik)I_{ik}(e^{x})$ and $\tilde{f}_{-}(x,k)=\tilde{f}(x,-k)$ having the asymptotics
$$\tilde{f}_{\pm}(x,k)=e^{\pm ikx} + o(1)\quad\text{as}\quad x\rightarrow-\infty.$$
Here
$$\tilde{\vphi}(x,k)=\tilde{M}(k)\tilde{f}_{+}(x,k) + \tilde{M}(-k)\tilde{f}_{-}(x,-k),$$
where
$$\tilde{M}(k)=-\frac{\pi\,2^{-1-ik}}{\sin(\pi i k)\Gamma(1+ik)}=2^{-1-ik}\Gamma(-ik).$$

Using these arguments and Remark \ref{K-bessel}, consider the solution of equation  \eqref{e.v.-m}, obtained from  $\hat{\vphi}(p,k)$ by
multiplying by the quasi-constant $\sinh\!\left(\frac{\pi i p}{\omega'}\right) + \sinh\!\left(\frac{\pi i k}{\omega'}\right)$, and for real $x$ and $k$ put
\begin{equation} \label{f-coor}
f(x,k)=\frac{1}{2 \sinh\!\left(\frac{\pi i k}{\omega'}\right)
\!M(k)}\int\limits_{L}\hat{\vphi}(p,k)\left(\sinh\left(\frac{\pi i p}{\omega'}\right) + \sinh\left(\frac{\pi i k}{\omega'}\right)\right)e^{2\pi i px}dp.
\end{equation}
The next statement shows that the functions $f_{+}(x,k)=f(x,k)$ and $f_{-}(x,k)=f(x,-k)$ indeed play the role of the Jost solutions of equation  \eqref{ev-H}.

\begin{lemma} \label{f-properties}
The following statements hold.
\begin{itemize} 
\item[(i)] For real $x$ and  $k$ the functions $f_{\pm}(x,k)$ have the following asymptotics as $x\rightarrow-\infty$ 
$$f_{\pm}(x,k)=e^{\pm2\pi ikx} + o(1).$$
\item[(ii)] For real $x$ the functions $f_{\pm}(x,k)$ admit analytic continuation into the physical strip  $0< \im k\leq |\omega|$ and satisfy 
$$\overline{f_{\pm}(x,k)}=f_{\pm}(x,-\bar{k}).$$
\item[(iii)] For fixed $k$ in the physical strip the functions $f_{\pm}(x,k)$ are entire functions of the variable  $x$ and satisfy equation\eqref{ev-H}. Asymptotics in part (i) remain valid in the strip $0\leq\im x\leq 2|\omega'|$ as well.
\item[(iv)] The following relation holds
$$\vphi(x,k)=M(k)f_{+}(x,k)+M(-k)f_{-}(x,k).$$
\item[(v)] The following estimates hold
$$|f_{\pm}(x,k)|\leq Ce^{\mp2\pi\kappa x},$$
uniformly on $-\infty<x\leq a$, where  $0\leq \kappa=\im k\leq |\omega|$ and
$$|f_{\pm}(x,k)|\leq Ce^{2\pi(|\omega|-|\omega'|)x},\quad |f_{\pm}(x+2\omega',k)|\leq Ce^{2\pi(|\omega|+|\omega'|)x},$$
uniformly on $a\leq x<\infty$.
\end{itemize}
\end{lemma}
\begin{proof} To proof (i) it is sufficient to shift the contour of integration in  \eqref{f-coor} for negative $x$ to the lowe half-plane and to use regularity of the integrand at $p=-k$ (due to the multiplication by the quasi-constant $\sinh\!\left(\frac{\pi i p}{\omega'}\right) + \sinh\!\left(\frac{\pi i k}{\omega'}\right)$). Parts (ii)--(iv) immediately follow from representation  \eqref{f-coor}, rewritten in the form
\begin{equation} \label{f-phi}
f(x,k)=\frac{1}{4 \sinh\!\left(\frac{\pi i k}{\omega'}\right)\!M(k)}\left(\vphi(x-2\omega,k)-\vphi(x+2\omega,k)+2\sinh\!\left(\tfrac{\pi i k}{\omega'}
\right)\!\vphi(x,k)\right),
\end{equation}
and similar properties of the function $\vphi(x,k)$ in Lemma \ref{phi-properties}. Since the  `potential' $e^{\frac{\pi i x}{\omega}}$ in equation  \eqref{ev-H} has period $2\omega$,  the functions $\vphi(x\pm2\omega,k)$ also satisfy \eqref{ev-H}. The proof of the estimates  (v) goes in a  standard way using integral representation  \eqref{f-coor}.
\end{proof}

\begin{remark}  Formula  \eqref{f-phi} is a difference analog of \eqref{I-K}. The function $f(x,k)$ is not invariant under the interchange of $\omega$ and $\omega'$, and therefor does not satisfy the duel eigenvalue problem(cf. Remark \ref{dual-2}).
\end{remark}
\begin{remark} In the case when $\im\tau>0$, the Jost solutions  $f_{\pm}(x,k)$ can be defined using power series in the variable  $e^{\frac{\pi ix}{\omega}}$, which are absolutely convergent for all $x\in\RR$. In our case  $\tau=b^{2}>0$ and the problem of small denominators emerges, so that the corresponding series are no longer convergent for all $x$. This is why we are using integral representation \eqref{f-coor}.
\end{remark}
\subsection{The Casorati determinant} As in Section \ref{res-free} it is checked directly that the Casorati determinant 
$$C(f,g)(x,k)=f(x+2\omega',k)g(x,k)-f(x,k)g(x+2\omega',k)$$
of two solutions of equation  \eqref{ev-H} is a  $2\omega'$-periodic function of the variable  $x$. Unlike its continuous analog, the Wronskian, the Casorati determinant, generally speaking, is no longer a constant.  Nevertheless, the following statement holds.\begin{lemma} \label{Casorati} We have the formula
$$C(f_{-}, f_{+})(x,k)= 2\sinh\left(\frac{\pi i k}{\omega}\right).$$
\end{lemma}
\begin{proof} Put $C(x)=C(f_{-}, f_{+})(x,k)$. As $x\rightarrow -\infty$, it follows from (i) and (iii) in Lemma \ref{f-properties} that in the strip $0\leq \im x\leq 2|\omega'|$ the following asymptotic holds
$$C(x)=2\sinh\left(\frac{\pi i k}{\omega}\right) + o(1).$$
When $x\rightarrow\infty$, using 
$$C(x)=2\sinh\left(\frac{\pi ik}{\omega'}\right)C(f_{-},\vphi)(x,k)$$
and the estimates in parts (iv) of Lemma \ref{phi-properties} and (v) of Lemma \ref{f-properties}, we obtain that on the lines $\im x=0$ and $\im x=2|\omega'|$ the function $C(x)$ is bounded. Further, it follows from integral representation \eqref{f-coor} that the function $C(x)$ has no more than exponential growth as $x\rightarrow\infty$. Using Phragm\'{e}n-Lindel\"{o}f theorem, we conclude that $2\omega'$-periodic function $C(x)$ is bounded in the strip $0\leq \im x\leq 2|\omega'|$. Therefore, $C(x)= 2\sinh\left(\frac{\pi i k}{\omega}\right)$. 
\end{proof}

\section{The eigenfuncion expansion theorem}
\subsection{The resolvent of the operator $H$} \label{resolvent-H} Consider the integral operator $R(\lambda)$ on $L^{2}(\RR)$ with the integral kernel
\begin{equation} \label{R-kernel-2}
R(x,y;\lambda)=\frac{\omega}{\sinh\left(\frac{\pi i k}{\omega}\right)M(k) }\left(\frac{f_{-}(x,k)\vphi(y,k)}{1-e^{\frac{\pi i}{\omega'}(x-y)}}+\frac{f_{-}(y,k)\vphi(x,k)}{1-e^{-\frac{\pi i}{\omega'}(x-y)}} \right),
\end{equation}
so that
$$R(y,x;\lambda)=R(x,y;\lambda)\quad\text{and}\quad \overline{R(x,y;\lambda)}=R(x,y;\bar\lambda).$$
The following statement holds.
\begin{proposition} \label{resolv} The operator $R(\lambda)$ 
for $\lambda\in\CC\setminus [2,\infty)$ is the resolvent of the operator $H$. In other words, 
$R(\lambda)=(H-\lambda I)^{-1}.$ 
\end{proposition}
\begin{proof} As in the case of the operator $H_{0}$ (see Remark \ref{H-0-res}), the statement that for smooth function $g(x)$ with compact support the function
$$\psi(x)=\int_{-\infty}^{\infty}R(x,y;\lambda)g(y)dy\in D(H)$$
and satisfies equation $(H-\lambda I)\psi=g$, reduces to the verification of the following equation
\begin{equation} \label{delta2}
R(x+2\omega'-i0,y;\lambda)+R(x-2\omega'+i0,y;\lambda)+(e^{\frac{\pi i x}{\omega}}-\lambda)R(x,y;\lambda)=\delta(x-y),
\end{equation}
understood in the distributional sense. As in Remark \ref{H-0-res}, the functions $\vphi(x,k)$ and $f_{-}(x,k)$ satisfy 
equation  \eqref{ev-H}, so that the distribution in the left hand side of equation \eqref{delta2} has support only at $x=y$, and its singular part coincides with the singular part of the function 
\begin{gather*}
-\frac{\omega\omega'}{\pi \sinh\left(\frac{\pi i k}{\omega}\right)M(k) }\left(\frac{f_{-}(x+2\omega',k)\vphi(y,k)-f_{-}(y,k)\vphi(x+2\omega',k)}{x-y-i0} \right.\\
\left.+\frac{f_{-}(x-2\omega',k)\vphi(y,k)+f_{-}(y,k)\vphi(x+2\omega',k)}{x-y+i0}\right)
\end{gather*}
in the neighborhood of $x=y$. As in the derivation in Remark \ref{H-0-res}, equation \eqref{delta2}  follows from the formula 
$$C(f_{-},\vphi)(x,k)=2\sinh\left(\frac{\pi i k}{\omega}\right)M(k),$$
which, in turn, follows from Lemma  \ref{Casorati}, and from the Sokhotski-Plemelj formula.

It remains to show that the kernel  \eqref{R-kernel-2} when $\lambda\in\CC\setminus [2,\infty)$ defines a bounded operator on в $L^{2}(\RR)$. This immediately follows from the estimate
$$|R(x,y;\lambda)|\leq Ce^{-2\pi\kappa|x-y|},\quad\kappa=\im k,$$
which is a consequence of the estimates in Lemmas  \ref{phi-properties} and \ref{f-properties}. Indeed, since $R(x,y;\lambda)=R(y,x;\lambda)$, it is sufficient to assume that  $y\leq x$. Let us rewrite \eqref{R-kernel-2} in the form
$$R(x,y;\lambda)=\frac{\omega\left(f_{-}(x,k)\vphi(y,k)e^{2\pi i\omega(x-y)} - f_{-}(y,k)\vphi(x,k)e^{-2\pi i\omega(x-y)} \right)}{2\sinh 2\pi i\omega(x-y)\sinh\left(\frac{\pi i k}{\omega}\right)M(k)}$$
and consider first the case $0\leq y\leq x$. Using the estimates from part (iv) of Lemma \ref{phi-properties} and  part (v) of Lemma  \ref{f-properties}, we get
\begin{align*}
|R(x,y;\lambda)|\leq Ce^{-2\pi|\omega|(x-y)}\left(e^{2\pi(|\omega|-|\omega'|)x}e^{-2\pi(|\omega|+|\omega'|)y}e^{-2\pi|\omega|(x-y)}+\right.\\
\left.+e^{-2\pi(|\omega|+|\omega'|)x}e^{-2\pi\kappa y}e^{2\pi|\omega|(x-y)}\right)
\leq 2Ce^{-2\pi|\omega|(x-y)}.
\end{align*} 
For the case $y<0\leq x$ we have
\begin{align*}
|R(x,y;\lambda)|\leq Ce^{-2\pi|\omega|(x-y)}\left(e^{2\pi(|\omega|-|\omega'|) x}e^{-2\pi\kappa y}e^{-2\pi|\omega|(x-y)}+\right.\\
\left.+e^{-2\pi(|\omega|+|\omega'|)x}e^{2\pi\kappa y}e^{2\pi|\omega|(x-y)}\right)\\
\leq C\left(e^{2\pi(|\omega|-\kappa)y}e^{-2\pi|\omega|(x-y)}+e^{-2\pi|\omega|x}e^{2\pi\kappa y}\right)\leq 2Ce^{-2\pi\kappa(x-y)}.
\end{align*} 
And, finally, in the remaining case  $y\leq x<0$ we have the estimate
\begin{gather*}
|R(x,y;\lambda)|\leq Ce^{-2\pi|\omega|(x-y)}\left(e^{2\pi\kappa x}e^{-2\pi\kappa y}e^{-2\pi|\omega|(x-y)}+\right.\\
\left.+e^{-2\pi\kappa x}e^{2\pi\kappa y}e^{2\pi|\omega|(x-y)}\right)\leq 2Ce^{-2\pi\kappa(x-y)}.
\qedhere
\end{gather*} 
\end{proof}
\begin{remark}
Formula \eqref{R-kernel-2} can be also used as a definition of the operator $H$. 
\end{remark}

\subsection{The eigenfunction expansion} \label{expansion}
Explicit formula  \eqref{R-kernel-2} for the resolvent $R(\lambda)$ immediately leads to the eigenfunction expansion theorem for the operator $H$. Namely, denote by $E(\Delta)$, where $\Delta$ is a Borel subset of  $\RR$, the resulution of the identity for the self-adjoint operator  $H$ (see \cite{AG, DS}). When there is no point spectrum we have the formula
\begin{equation*}
E(\Delta) =\lim_{\varepsilon\rightarrow 0^{+}}\frac{1}{2\pi i}\int_{\Delta}\left(R(\lambda+i\varepsilon)-R(\lambda-i\varepsilon)\right)d\lambda
\end{equation*}
(see \cite[Ch.~XII]{DS}), sometimes called the Stone's formula.
In particular, putting  $\Delta=[2,\infty)$, for the operator $H$ we get
\begin{equation} \label{Stone}
I =\lim_{\varepsilon\rightarrow 0^{+}}\frac{1}{2\pi i}\int_{2}^{\infty}\left(R(\lambda+i\varepsilon)-R(\lambda-i\varepsilon)\right)d\lambda.
\end{equation}
It is this formula that serves as a base for the derivation of the eigenfunction expansion theorem. 
\begin{theorem} \label{theorem-1} The following statements hold.
\begin{itemize}
\item[(i)]
The operator $\cU$, given by the formula
$$(\cU\psi)(k)=\int_{-\infty}^{\infty}\psi(x)\varphi(x,k)dx,$$
isometrically maps $L^{2}(\RR)$ onto the Hilbert space $\cH_{0}=L^{2}([0,\infty), \rho(k)dk)$ with the spectral function 
$$\rho(k)=\frac{1}{|M(k)|^{2}}=4\sinh\!\left(\frac{\pi i k}{\omega}\right)\sinh\!\left(\frac{\pi i k}{\omega'}\right).$$
In other words,  $\cU: L^{2}(\RR)\rightarrow \cH_{0}$ and wherein
$$\cU^{*}\cU=I\quad\text{and}\quad \cU\cU^{*}=I_{0},$$
where $I_{0}$ is the identity operator on $\cH_{0}$.
\item[(ii)] The operator $\cU H\cU^{-1}$ is a multiplication by the function $2\cosh\!\left(\frac{\pi i k}{\omega}\right)$ operator on $\cH_{0}$, so that the operator $H$ has simple absolutely continuous spectrum filling $[2,\infty)$. 
\end{itemize}
\end{theorem}

\begin{proof} We will show that  for each $\psi(x)\in\cD$ the following identity holds
\begin{equation} \label{completeness}
\psi(x)=\int_{0}^{\infty}\left(\int_{-\infty}^{\infty}\psi(y)\vphi(y,k)dy\right)\vphi(x,k)\rho(k)dk.
\end{equation}
Using equation  \eqref{ev-H} for $\vphi(x,k)$ we get that $(\cU\psi)(k)$ as $k\rightarrow\infty$ decays faster than any power of $e^{-\frac{\pi i k}{\omega}}$, so that all integrals are obviously absolutely convergent. Formula \eqref{completeness} can be proved by the method of complex integration as in  \cite[\S 2]{F59} (see also \cite[Ch.~3]{T1}), as well as by applying formula  \eqref{Stone}, which is what we will using here. Namely, apply \eqref{Stone} to the function $\psi(x)\in\cD$. Explicitly computing the jump of the resolvent kernel $R(x,y;\lambda)$ across the continuous spectrum using relation (iv) of Lemma  \eqref{f-properties}, we have
\begin{align*}
R(x,y;\lambda+i0)-R(x,y;\lambda-i0) & =\frac{\omega}{\sinh\left(\frac{\pi i k}{\omega}\right)}\frac{\varphi(x,k)\varphi(y,k)}{|M(k)|^{2}}\\
&=\frac{\omega}{\sinh\left(\frac{\pi i k}{\omega}\right)}\vphi(x,k)\vphi(y,k)\rho(k),
\end{align*}
where it was used that the case $\lambda +i0$ corresponds to the variable $k>0$, and the case $\lambda-i0$ --- to the variable $-k$.
Using $d\lambda=\frac{2\pi i}{\omega}\sinh\left(\frac{\pi i k}{\omega}\right)dk$, we arrive at \eqref{completeness}, and by multiplying it by $\overline{\psi(x)}$ and integrating, we obtain
$$\|\psi\|_{L^{2}(\RR)}^{2}=\|\cU\psi\|_{\cH_{0}}^{2}.$$
(the change of order of integration is legitimate because of the Fubini theorem). Thus the operator $\cU$, defined on a dense linear subset $\cD\subset L^{2}(\RR)$, maps it into the Hilbert space $\cH_{0}$ and is an isometry. Therefore $\cU$ admits an isometric extension to the whole  $L^{2}(\RR)$, which proves the  \emph{completeness relation} 
$$\cU^{*}\cU=I.$$

The \emph{orthogonality relation} 
$$\cU\cU^{*}=I_{0}$$
is equivalent to the statement that the image of the operator $\cU$ in $\cH_{0}$, the closed subspace $\im\cU$, coincides with $\cH_{0}$.
This is proved in the standard way (see \emph{e.g.}, \cite[Ch.~3]{T1}). Namely, on the domain $\cD$ we have $\cU (H-\lambda I)=(\hat{H}-\lambda I) \cU$, where $\hat{H}$ is the multiplication by  $2\cosh(\frac{\pi ik}{\omega})$ operator on $\cH_{0}$. From here we get
$$\cU R(\lambda) =\hat{R}(\lambda)\cU,$$
where $\hat{R}(\lambda)$ is the resolvent of the operator $\hat{H}$. Thus $\im\cU$ is an invariant subspace for the operator  $\hat{R}(\lambda)$ for all $\lambda\in\CC\setminus [2,\infty)$. Therefore, $\hat{R}(\lambda)$ commutes with the orthogonal projection operator  $P$ onto the subspace $\im\cU$. Hence $P$ is a function of the operator $\hat{H}$ which, in turn, is the function of of the multiplication operator by the variable $k$ on $\cH_{0}$. Thus we obtain that $P$ is a multiplication operator by a characteristic function $\chi_{\Delta}$ of some Borel subset  $\Delta$ in $[0,\infty)$. On the other hand, if for some $k>0$
$$\int_{-\infty}^{\infty}\psi(x)\varphi(x,k)dx=0$$
for all $\psi(x)\in C_{0}(\RR)$, then $\vphi(x,k)=0$ for all $x$, so that necessarily $\Delta=[0,\infty)$. 

This completes the proof of (i). The part (ii) follows from the above arguments. 
\end{proof}

\begin{remark} In the physics literature the completeness and orthogonality relations, understood in the distributional sense, are written as follows 
\begin{equation*}
\int_{0}^{\infty}\vphi(x,k)\vphi(y,k)\rho(k)dk=\delta(x-y)
\end{equation*}
and
\begin{equation*}
\int_{-\infty}^{\infty}\vphi(x,k)\vphi(x,l)dx=\frac{1}{\rho(k)}\delta(k-l), \quad k,l>0.
\end{equation*}
As in the case of one dimensional Schr\"{o}dinger equation (see \emph{e.g.}, \cite[Ch.~3]{T1}), the last relation can be proved directly by using the Casorati determinant 
Namely, put $C(x)=C(\vphi(x,k),\vphi(x,l))$ and integrate this function over the contour $D$ on Fig. 3.

\vspace{5mm}
\begin{tikzpicture}[thick, scale=1.2]
\draw (-4,0) -- (4,0);
\draw (4,0) -- (4,-1);
\draw (4,-1) -- (-4,-1);
\draw (-4,-1) -- (-4,0);
\node at (0,.5) {$D$};
\node at (-4,.3) {$-N$};
\node at (0,0) {$>$};
\node at (0,-1) {$<$};
\node at (-4,-.5) {$\wedge$};
\node at (4,-.5) {$\vee$};
\node at (4,.3) {$N$};
\node at (-4,-1.3) {$-N-2\omega'$};
\node at (4,-1.3) {$N-2\omega'$};
\node at (0,-2) {$\text{Fig. 3}$};
\end{tikzpicture}

\noindent
By Cauchy theorem,
$$\int\limits_{D}C(x)dx=0.$$
On the other hand, using the formula
$$C(x)-C(x-2\omega')=(\lambda-\mu)\vphi(x,k)\vphi(x,l),$$
where $\lambda=2\cosh\!\left(\frac{\pi i k}{\omega}\right)$ and $\mu=2\cosh\!\left(\frac{\pi i l}{\omega}\right)$, we get the equation
$$\int_{-N}^{N}\vphi(x,k)\vphi(x,l)dx= \frac{1}{(\lambda-\mu)}\left(\int_{N-2\omega'}^{N}C(x)dx - \int_{-N-2\omega'}^{-N}C(x)dx\right).$$
It follows form the estimates (iv) in Lemma \ref{phi-properties} that the first integral is exponentially decaying as $N\rightarrow\infty$. Using   (i) and  (iv) in Lemma \ref{f-properties}, well-known formula from the distribution theory 
$$\lim_{N\rightarrow\infty}\frac{\sin 2\pi(k-l)N}{k-l}=\pi\delta(k-l),$$
and the Riemann-Lebesgue lemma, we obtain
$$\lim_{N\rightarrow\infty} \frac{1}{(\mu-\lambda)}\int_{-N-2\omega'}^{-N}C(x)dx=\frac{1}{\rho(k)}\delta(k-l).$$
\end{remark}

\begin{remark} \label{K-L} Because $W(J_{\nu},K_{\nu})=-1$ (see Remark \ref{I-bessel}), the resolvent kernel  $\tilde{R}(\lambda)$ of the operator $\displaystyle{\tilde{H}=-\frac{d^{2}}{dx^{2}}+e^{2x}}$
has the form
$$\tilde{R}(x,y;\lambda)=\frac{1}{2ik\tilde{M}(k)}(\tilde{f}_{-}(x,k)\tilde{\vphi}(y,k)\theta(y-x)+\tilde{f}_{-}(y,k)\tilde{\vphi}(x,k)\theta(x-y)),$$
where $\im k>0$ (see Remark \ref{schrodinger}).
As in Theorem \ref{theorem-1}, we get that the operator $\tilde{\cU}$, given by the formula
$$(\tilde{\cU}\psi)(k)=\int_{-\infty}^{\infty}\psi(x)\tilde{\varphi}(x,k)dx,$$
isometrically maps $L^{2}(\RR)$ onto $\tilde{\cH}_{0}=L^{2}([0,\infty),\tilde{\rho}(k)dk)$,  where
$$\tilde{\rho}(k)=\frac{2\pi}{|\tilde{M}(k)|^{2}}=\frac{2k\sinh\pi k}{\pi^{2}}.$$
Here the operator $\tilde{\cU}\tilde{H}\tilde{\cU}^{-1}$ is a multiplication operator by $k^{2}$ on $\tilde{\cH}_{0}$. The formulas
$$\tilde{\psi}(k)=\int_{-\infty}^{\infty}\psi(x)K_{ik}(e^{x})dx$$
and
$$\psi(x)=\frac{2}{\pi^{2}}\int_{0}^{\infty}\tilde{\psi}(k)K_{ik}(e^{x})k\sinh\pi kdk,$$
after the change of variable $x=\ln y$, 
are known in the theory of special functions as Kontorovich-Lebedev transform and its inverse (see~\cite[Ch.~XII]{BE2}), and the equality 
$$\int_{-\infty}^{\infty}|\psi(x)|^{2}dx=\frac{2}{\pi^{2}}\int_{0}^{\infty}|\tilde{\psi}(k)|^{2}k\sinh\pi kdk$$
--- as the Parseval theorem. The eigenfunction expansion theorem for the operator $\tilde{H}$ gives a spectral interpretation of the Kontorovich-Lebedev transform, so that Theorem  \ref{theorem-1} may be considered as a  $q$-analog of this transformation.
\end{remark}

\subsection{The scattering theory} \label{scattering}  Here we briefly outline the scattering theory for the operator $H$.
Put
$$\vphi^{(+)}(x,k)=\frac{1}{M(k)}\vphi(x,k).$$
We have as  $x\rightarrow-\infty$,
$$\vphi^{(+)}(x,k)=e^{2\pi ikx} + S(k)e^{-2\pi ikx} + o(1),$$
where
$$S(k)=\frac{M(-k)}{M(k)}=e^{-4\pi i\omega''k}\,\frac{\gamma(-2k-\omega'')}{\gamma(2k-\omega'')}.$$
According to the stationary scattering theory (see \cite{F59,F74}), multiplication by the function $S(k)$ operator plays the role of the scattering operator on the space $\cH_{0}$ as well as on $L^{2}([0,\infty))$. Note that the operator $\cU^{(+)}$, defined as 
$$(\cU^{(+)}\psi)(k)=\int_{-\infty}^{\infty}\psi(x)\vphi^{(+)}(x,k)dx,$$
isometrically maps $L^{2}(\RR)$ onto $L^{2}([0,\infty))$. As in \cite{FD, Faddeev-2}, it is convenient to interpret the latter space as a subspace in $L^{2}(\RR)$ of functions $\chi(k)$ satisfying
$$\chi(k)=S(k)\chi(-k).$$
The operator  $\cU^{(-)}$ is defined similarly by using the solution $\vphi^{(-)}(x,k)=\overline{\vphi^{(+)}(x,k)}$. The operators $\cU^{(\pm)}$ are used for the non stationary approach to the scattering theory (see~\cite{F59,F74}), and we leave its formulation to the reader. 
\begin{remark} In a similar way one formulates the scattering theory for the operator $\tilde{H}$. Here for the scattering operator $\tilde{S}$ we have
$$\tilde{S}(k)=\frac{\tilde{M}(-k)}{\tilde{M}(k)}=-2^{2ik}\frac{\Gamma(1+ik)}{\Gamma(1-ik)}$$
(cf. formula (5.19) in \cite{ZZ-1}).
\end{remark}


\begin{thebibliography}{99}
\bibitem[1]{F59} L.D. Faddeev, ``Inverse problem of quantum scattering theory'', \emph{Uspekhi Matem. Nauk}, \textbf{14}:4(88) (1959), 57--119 (Russian); English translation in \emph{J. Math. Phys.} \textbf{4}:1 (1963), 72--104.
\bibitem[2]{F74} L.D. Faddeev, ``Inverse problem of quantum scattering theory. II'', \emph{Itogi Nauki i Tekhniki. Sovremennye Problemy Matematiki}, \textbf{3}, VINITI, Moscow, (1974), 93--180 (Russian); English translation in \emph{J. of Soviet Math.}, \textbf{5}:3 (1976), 334--396.
\bibitem[3]{BPZ} A.A. Belavin, A.M. Polyakov, A.B. Zamolodchikov, ``Infinite conformal symmetry in two-dimensional quantum field theory'', \emph{Nucl. Phys. B}, \textbf{241}:2 (1984), 333--380.
\bibitem[4]{FT1} L.D. Faddeev and L.A. Takhtajan, ``Liouville model on the lattice'', {\em Field theory, quantum gravity and strings ({M}eudon/{P}aris,1984/1985)},  {\em Lecture Notes in Phys.}, \textbf{246} (1986), 166--179.
\bibitem[5]{FC} V.V. Fock, L.O. Chekhov ``A quantum Teichm\"{u}ller space'', \emph{Teoret. and Math. Phys}, \textbf{120}:3 (1999), 511--528 (Russian); English translation in \emph{Theor. Math. Phys.}, \textbf{120}:3 (1999), 1245--1259.
\bibitem[6]{K1} R.~Kashaev, ``The quantum dilogarithm and {D}ehn twists in quantum {T}eichm\"uller theory'', {\em Integrable structures of exactly solvable two-dimensional models of quantum field theory ({K}iev, 2000)}, {\em NATO Sci. Ser. II Math. Phys. Chem.}, \textbf{35} (2001), 211--221. 
\bibitem[7]{PT} B.~Ponsot and J.~Teschner, ``Clebsch-Gordan and Racah-Wigner coefficients for a continuous series of representations of $U_{q}(\mathrm{sl}(2,\RR))$'', \emph{Commun. Math. Phys.}, \textbf{224}:3 (2001), 613--655.
\bibitem[8]{FD} S.E. Derkachov, L.D. Faddeev, ``$3j$-symbol for the modular double of $\mathrm{SL}_{q}(2,\RR)$ revisited'', arXiv:13025400, 2013.
\bibitem[9]{T1} Leon A. Takhtajan, \emph{Quantum mechanics for mathematicians}, Graduate Studies in Mathematics \textbf{95}, Amer. Math. Soc., Providence, RI,  2008.
\bibitem[10]{ZZ-1} A.B. Zamolodchikov, Al.B. Zamolodchikov, ``Conformal bootstrap in Liouville field theory'', \emph{Nuclear Phys. B},  \textbf{477}:2 (1996), 577--605.
\bibitem[11]{Faddeev-1}L.D. Faddeev, ``Modular double of a quantum group'', \emph{Conf\'{e}rence Mosh\'{e} Flato 1999, Vol. I (Dijon)}, 149--156, \emph{Math. Phys. Stud.}, \textbf{21}, Kluwer, 2000.
\bibitem[12]{Sh} T. Shintani, ``On a Kronecker limit formula for real quadratic field'', \emph{J. Fac. Sci. Univ. Tokyo Sect. IA Math.}, \textbf{24}:1 (1977), 167--199.
\bibitem[13]{Kurokawa} N. Kurokawa, ``Multiple sine functions and Selberg zeta functions'', \emph{Proc. Japan Acad. Ser. A Math. Sci.}, \textbf{67}:3 (1991), 61--64.
\bibitem[14]{Ruij} S.N.M. Ruijsenaars, ``First order analytic difference equations and integrable quantum systems'', \emph{J. Math. Phys.}, \textbf{38}:2 (1997), 1069--1146.
\bibitem[15]{ZZ-2} Alexander B. Zamolodchikov, Alexey B. Zamolodchikov,  ``Factorized $S$-matrices in two dimensions as the exact solutions of certain relativistic quantum field theory models``,'' \emph{Ann. Physics}, \textbf{120}:2 (1979), 253--291.
\bibitem[16]{Smirnov} F.A. Smirnov, \emph{Form Factors in Completely Integrable Models of Quantum Field Theory}, World Scientific, 1992.
\bibitem[17]{Barnes} E.W. Barnes, ``Genesis of the double gamma function'', \emph{Proc. London Math. Soc.}, \textbf{31} (1899) 358-381.
\bibitem[18]{Al} V.P. Alekseevksi, ``On functions similar to the function gamma'', \emph{Proceedings of the Kharkov Math. Society}, (2) \textbf{1} (1889), 169-238 (in Russian).
\bibitem[19]{Volkov} Alexandre Y. Volkov, ``Noncommutative hypergeometry'', \emph{Commun. Math. Phys.}, \textbf{258}:2 (2005), 257-273.
\bibitem[20]{BE} H. Bateman and A. Erd\'{e}lei, \emph{Higher transcendental functions}, Vol. 2, McGraw-Hill, New York, 1953.
\bibitem[21]{AG} N.I. Akhiezer, I.M. Glazman, \emph{Theory of linear operators in Hilbert space}, Dover, 1993.
\bibitem[22]{DS} N. Danford and J.T. Schwartz, \emph{Linear operators Part II Spectral theory}, Wiley, 1988.
\bibitem[23]{BE2} H. Bateman and A. Erd\'{e}lei, \emph{Tables of integral transforms}, Vol. 2, McGraw-Hill, New York, 1954.
\bibitem[24]{Faddeev-2} L.D. Faddeev, ``Zero modes for the quantum Liouville model'', arXiv: 1404.1713, 2014.
\end{thebibliography}
\end{document}